\documentclass[12pt, a4paper]{amsart}
\pdfoutput=1
\usepackage{amsmath,amsfonts,amssymb,amsthm,color}  
\usepackage{mathrsfs}

\usepackage[utf8]{inputenc}
\usepackage{graphicx}

\theoremstyle{plain}
\newtheorem{theorem}{Theorem}[section]
\newtheorem{corollary}[theorem]{Corollary}
\newtheorem{lemma}[theorem]{Lemma}

\theoremstyle{definition}
\newtheorem{definition}[theorem]{Definition}

\newtheorem{remark}[theorem]{Remark}

\numberwithin{equation}{section}
\newtheorem*{theorem*}{Theorem}

\def\Xint#1{\mathchoice
{\XXint\displaystyle\textstyle{#1}}
{\XXint\textstyle\scriptstyle{#1}}
{\XXint\scriptstyle\scriptscriptstyle{#1}}
{\XXint\scriptscriptstyle\scriptscriptstyle{#1}}
\!\int}
\def\XXint#1#2#3{{\setbox0=\hbox{$#1{#2#3}{\int}$}
\vcenter{\hbox{$#2#3$}}\kern-.5\wd0}}

\def\dashint{\Xint-}

\newcommand{\dy}{\, \mathrm{d}y}
\newcommand{\dx}{\, \mathrm{d}x}

\DeclareMathOperator{\diam}{diam}

\DeclareMathOperator{\supp}{supp}

\DeclareMathOperator{\BMO}{BMO}
\DeclareMathOperator{\BV}{BV}

\providecommand{\abs}[1]{ \lvert#1  \rvert}
\providecommand{\norm}[1]{ \lVert#1  \rVert}

\title[Poincar\'e inequalities]{Poincar\'e inequalities for the maximal function}
\author{Olli Saari}

\begin{document}
\begin{abstract}
We study generalized Poincar\'e inequalities. We prove that if a function satisfies a suitable inequality of Poincar\'e type, then the Hardy-Littlewood maximal function also obeys a meaningful estimate of similar form. As a by-product, we get a unified approach to proving that the maximal operator is bounded on Sobolev, Lipschitz and BMO spaces. 
\end{abstract}

\address{Olli Saari,	
	Department of Mathematics and Systems Analysis, 
	Aalto University School of Science,
	FI-00076 Aalto, 
	Finland
} \email{olli.saari@aalto.fi} 

\subjclass[2010]{Primary: 42B25, 46E35, 42B35} 

\keywords{Maximal function,
			Poincar\'e inequality, 
			Sobolev space,
			BMO}

\thanks{The author is supported by the V\"ais\"al\"a Foundation.}

\maketitle


\section{Introduction}
It was proved by Kinnunen in \cite{Kinnunen1997} that the Hardy-Littlewood maximal function gives rise to a bounded operator
\[M : W^{1,p}(\mathbb{R}^{n}) \to  W^{1,p}(\mathbb{R}^{n}) , \quad 1< p \leq \infty.\]
Similar results have later been established for maximal functions restricted to domains \cite{KL1998}, for fractional maximal functions \cite{HLT2013,HKKT2015,KS2003}, and for certain convolution type maximal functions \cite{CS2013,CFS2015}. Continuity as well as action on some function spaces in the Triebel-Lizorkin scale have been studied in \cite{Korry2002,Luiro2007,Luiro2010}. It is also well known that the maximal function behaves well on $\BMO$ and spaces of H\"older continuous functions. See \cite{BdVS1981,Buckley1999}.

Many smoothness properties of functions can be characterized using generalized Poincar\'e inequalities. For instance, $u \in L^{1}_{loc}(\mathbb{R}^{n})$ has weak derivatives in $L^{1}_{loc}(\mathbb{R}^{n})$ if and only if there is a non-negative function $g \in L^{1}_{loc}(\mathbb{R}^{n})$ satisfying 
\[ \dashint_{Q} |u- u_{Q}| \dx \lesssim \diam(Q) \dashint_{Q} g \dx \]
for all cubes $Q$. See Section 2 for definitions. Replacing the functional 
\[a(Q) = \diam(Q) \dashint_{Q} g \dx \]
by more general \textit{fractional averages}
\[a(Q) = \diam(Q)^{\alpha} \frac{\mu(Q)}{|Q|} , \quad \alpha \in [0,1], \]
we recover characterizations of other spaces of regular functions. 

The present paper is devoted to studying the action of the Hardy-Littlewood maximal operator at the level of generalized Poincar\'e inequalities. Our main theorem states that many of them persist when we pass to the maximal function. See Theorem \ref{theorem:preserved} for the precise statement. As a consequence, we can include many results about the regularity of the maximal functions in a single theorem. Our proof is very geometric, and the approach is likely to extend to slightly more general metric spaces. However, we are concerned with $\mathbb{R}^{n}$ in the present paper. The result is also interesting from the point of view of more abstract Poincar\'e inequalities; see \cite{BKM2016,BM2015,FPW1998,HK2000,JM2013,MP1998}. We do not know what are the most general functionals $a(Q)$ to which our approach applies, but a list of examples and corollaries demonstrates why the result for fractional averages is remarkable.

In addition to finding a unified approach to different function spaces, there is one more motivation for studying Poincar\'e inequalities that the maximal function satisfies. Namely, it 
is an open problem whether $u \mapsto |\nabla Mu|$ is a bounded operator from $W^{1,1}(\mathbb{R}^{n})$ to $L^{1}(\mathbb{R}^{n})$ for $n > 1$. This is Question 1 in the paper \cite{HO2004} by Haj\l{}asz and Onninen, and it has attracted quite a lot of attention; see for instance \cite{BCHP2012,CH2012,CM2015} and the references therein. Progress on this problem has been restricted to dimension one, where the question was answered by Tanaka \cite{Tanaka2002} in positive. This was actually before \cite{HO2004} was published. Later Kurka \cite{Kurka2010} extended Tanaka's theorem for the centred maximal operator in dimension one.

Even if we do not know how to bound $ |\nabla Mu|$ from $W^{1,1}(\mathbb{R}^{n})$ to $L^{1}(\mathbb{R}^{n})$, it is still true that the derivatives of $Mu$ are measurable functions in a considerably large subset of $\mathbb{R}^{n}$. This requires, of course, a carefully chosen interpretation for the generalized derivatives. Strictly speaking, a function that is not locally integrable cannot be a weak derivative. However, the distributional derivative acting outside a certain (possibly empty) set of singularities is a function in $L^{1,\infty}(\mathbb{R}^{n})$. This reproduces a result of Haj\l{}asz and Onninen \cite{HO2004} ensuring local integrability of $\nabla Mu$ under the additional hypothesis $|\nabla u| \in L \log L$.

The previous partial result can be included in the list of corollaries that follow from our main theorem. We can also prove that the same result holds for functions of bounded variation, whose generalized gradients are merely Radon measures instead of measurable functions. What prevents us from solving the question of Haj\l{}asz and Onninen is the fact that our approach is based on dominating the gradient by the maximal function, which need not be locally integrable a priori. Better results on this problem would need an analysis more delicate than what is possible to carry out in our general setting.

The structure of the present paper is as follows. Section 2 introduces the notation and preliminary results. Section 3 contains the main theorem establishing that the Poincar\'e inequalities are preserved under the action of the maximal operator. Several corollaries are also discussed. Section 4 gives an application towards the study of the endpoint Sobolev space $W^{1,1}(\mathbb{R}^{n})$. Section 5 contains a few remarks on the extension of the results to the fractional maximal function.

\vspace{0.3cm}

\noindent \textit{Acknowledgement.} The author would like to thank Carlos P\'erez for suggesting a problem that lead to this paper. The author also wishes to thank Juha Kinnunen for enlightening discussions on Sobolev spaces.

\section{Notation and preliminaries}
We use standard notation in $\mathbb{R}^{n}$, $n > 1$. The letter $C$ denotes a constant only depending on uninteresting quantities. We do not keep track of numerical constants and the dependency on the dimension $n \geq 1$. If $a \leq C b$, we write $a \lesssim b$. For a measurable set $E \subset \mathbb{R}^{n}$ we denote by $|E|$ its Lebesgue measure. By a function we mean a measurable real valued function of $n$ real variables. If $u$ is a locally integrable function, then
\[u_E = \dashint_{E} u \dx = \frac{1}{|E|} \int_E u \dx . \]
The diameter of a bounded set $E$ is denoted by
\[\diam(E) = \sup \{ \abs{x-y}: x,y \in E \}.\]
Sometimes we only use positive parts of functions. Then we denote $u^{+} = 1_{\{u > 0\}} u$. 

By a cube we always mean a cube with sides parallel to coordinate axes. In addition, our cubes are open even though sets of measure zero do not matter in our considerations. For a cube $Q$, we denote its side length by $\ell(Q)$ and its center by $c(Q)$. For a positive constant $\lambda > 0$, we understand $\lambda Q$ to be the cube with center $c(Q)$ and side length $\lambda \ell(Q)$. 

We define the non-centred Hardy-Littlewood maximal function acting on locally integrable functions to be
\[M u (x)= \sup_{Q \ni x} \dashint_{Q} \abs{u} \dx \] 
where the supremum is over all cubes with sides parallel to coordinate axes. The centred maximal function uses the supremum with the restriction that the cubes must be centred at $x$. We do not introduce separate notation for it since most of our results holds for both of them. We mention in our statements which variant we are using. We need the basic fact that $M: L^{1}(\mathbb{R}^{n}) \to L^{1,\infty}(\mathbb{R}^{n})$ is bounded where
\[\norm{u}_{L^{p,\infty}(\mathbb{R}^{n})} = \sup_{\lambda > 0} \lambda | \{x \in \mathbb{R}^{n}: \abs{u(x)} > \lambda \} |^{1/p} , \quad 0<p< \infty.\]
Also the boundedness $L^{p,\infty} \to L^{p,\infty}$ with $p > 1$ is needed. For all necessary details about $L^{p,\infty}$, we refer to \cite{Grafakos2008}. For an exponent $p \in (1,\infty)$, we denote the conjugate exponent by $p' = p/(p-1)$.

A function $u$ is said to belong to the local Sobolev space $W^{1,p}_{loc}(\mathbb{R}^{n})$ for $ p \in [1,\infty)$ if both $u$ and its distributional gradient belong to $L^{p}_{loc}(\mathbb{R}^{n})$, that is, for all compact $K$
\[\int_{K}( |u|^{p} + |\nabla u|^{p} ) \dx < \infty . \]
The global Sobolev space $W^{1,p}(\mathbb{R}^{n})$ is defined similarly, but global integrability is required instead. We can norm this space by summing the $L^{p}$ norms of the function and its gradient. Sobolev spaces can also be defined on domains; we mean open and connected subsets of $\mathbb{R}^{n}$. Weak differentiability of Sobolev functions can be characterized through Poincar\'e type inequalities.

\begin{definition}
\label{def:poincare}
Let $u \in L^{1}_{loc}(\mathbb{R}^{n})$ and let $\mu$ be a locally finite positive Borel measure. We say that $u$ and $\mu$ satisfy Poincar\'e inequality with $\alpha$ if for all cubes $Q$ it holds
\[\left(\dashint_{Q}|u-u_Q |^{q} \dx \right)^{1/q}\leq \diam(Q)^{\alpha} \frac{\mu(Q)}{|Q|} .\]
Here $q\geq 1$ and $\alpha \geq 0$.
\end{definition}

This general form of a Poincar\'e type inequality can be used to characterize various function spaces. As we show that it is preserved under the action of the Hardy-Littlewood maximal operator, we get a unified approach to studying the maximal function on such spaces. We list some examples.
\begin{itemize}
\item Let $p \in [1,\infty]$. Then $u \in W^{1,p}(\mathbb{R}^{n})$ if and only if $u \in L^{p}(\mathbb{R}^{n})$ and it satisfies the inequality of Definition \ref{def:poincare} with $\alpha = 1$, $q = 1$ and $\mu \in L^{p}(\mathbb{R}^{n})$. See \cite{Hajlasz2003} and \cite{EG1991}. 
\item $u \in BV$ if and only if $u \in L^{1}(\mathbb{R}^{n})$ and it satisfies the inequality of Definition \ref{def:poincare} with $\alpha = 1$, $q = 1$ and $\mu$ a Radon measure. See for instance \cite{Miranda2003} and \cite{EG1991}. 
\item $u \in \Lambda(\alpha)$ (the space of $\alpha$-H\"older continuous functions) if and only if $u \in L_{loc}^{1}(\mathbb{R}^{n})$ and it satisfies the inequality of Definition \ref{def:poincare} with $\alpha \in (0,1)$, $q = 1$ and $\mu \in L^{\infty}(\mathbb{R}^{n})$. See \cite{Campanato1963} and \cite{Meyers1964}.
\item $u \in \BMO$ if and only if $u \in L_{loc}^{1}(\mathbb{R}^{n})$ and it satisfies the inequality of Definition \ref{def:poincare} with $\alpha = 0$, $q = 1$ and $\mu \in L^{\infty}(\mathbb{R}^{n})$. This is just the definition of $\BMO$.
\end{itemize} 
We emphasize that Poincar\'e inequalities encode the local behaviour. For example, local Sobolev spaces are nested and decreasing with $p$. We can always work with the Poincar\'e to find out whether the function has a weak derivative, but it is another story if it belongs to the correct $L^{p}$ space, locally or globally. Moreover, the modulus of the weak gradient of a Sobolev function $u$ is, up to a dimensional constant, the minimal $\mu$ that can be inserted to the Poincar\'e inequality. This is proved along with the characterization in \cite{Hajlasz2003}.

All of the previous examples share the common feature of exhibiting a self-improving property. Namely, if the inequalities above hold with $q = 1$, then they also hold with some $q > 1$. The classical instances of this phenomenon are known as Sobolev-Poincar\'e and John-Nirenberg inequalities. However, similar phenomena also occur in the very general setting of $\mu$ being a locally finite Borel measure and beyond. This is sometimes called Franchi-P\'erez-Wheeden self-improvement.

\begin{lemma}[Franchi, P\'erez, and Wheeden \cite{FPW1998}]
\label{lemma:fpw}
Suppose that $u$, $\mu$ and $\alpha \in (0,1]$ satisfy Definition \ref{def:poincare} with $q = 1$. Then there exists $C > 0$ not depending on $u$ such that for all cubes $Q$ it holds
\[\frac{\norm{ 1_{Q}(u-u_{Q}) }_{L^{q,\infty}(\mathbb{R}^{n})}}{|Q|^{1/q}} \leq C \diam(Q)^{\alpha} \frac{\mu(Q)}{|Q|}\]
where $q = n/(n- \alpha)$.
\end{lemma}
\begin{proof}
This is just a special case of Theorem 2.3 in \cite{FPW1998}. The correct value of $q$ is given in Example 2.2 of that paper. The fact that we are in $\mathbb{R}^{n}$ with cubes saves us from dilating the cube in the right side, as pointed out in Remark 2.6 of \cite{FPW1998}. The constant $C$ does not depend on $u$ because the possible constants in Definition \ref{def:poincare} are hidden in the measure $\mu$. 
\end{proof}

\section{Poincar\'e inequality and maximal function}
The fact that the maximal function preserves some Poincar\'e inequalities is a consequence of two phenomena. The first one is the self-improvement of local integrability. Concrete instances of this phenomenon are the John-Nirenberg inequality and the Sobolev-Poincar\'e inequality. Self-improvement of local integrability is important when working at small scales. So-called chaining arguments and the geometry of the underlying space become dominating at large scales. Similar ideas have appeared separately in the literature, but the way in which we interpret and combine them is new. See \cite{BdVS1981,HMV2014,macmanus2002,MS2016}.

In some sense, it is possible to work with Sobolev spaces $W^{1,p}$ so that there is no need worry about what we call the local part. This is based on pointwise characterizations. See \cite{macmanus2002} and \cite{HO2004}. Similarly, the role of the non-local part is almost negligible in the case of $\BMO$. Compare to \cite{BdVS1981} and the slightly more complicated setting of \cite{Saari2016}. The following theorem finds a connection between these extreme cases at the level of generalized Poincar\'e inequalities. The assumptions of the theorem are motivated by the list of examples given after Definition \ref{def:poincare}. 

\begin{theorem}
\label{theorem:preserved}
Denote by $M$ the Hardy-Littlewood maximal function, centred or non-centred. Let $u \in L^{1}_{loc}(\mathbb{R}^{n})$ be a positive function such that $Mu \in L^{1}_{loc}(\mathbb{R}^{n})$. Suppose that there is a constant $C > 0$, uniform for all cubes $Q$, such that
\begin{equation}
\label{eq:poincare_general}
\dashint_{Q} \abs{u -u_{Q}} \dx \leq C \diam(Q)^{\alpha}  \frac{\mu(Q)}{|Q|}  
\end{equation}
where we freeze $\alpha$ and $\mu$ to one of the following two alternatives. Either
\begin{itemize}
\item $\alpha = 0$ and $\mu$ equals the Lebesgue measure, or
\item $\alpha \in (0,1]$ and $\mu$ is a locally finite positive Borel measure.
\end{itemize}
Then there is a constant $C$ comparable to the constant appearing in \eqref{eq:poincare_general} such that
\[ \dashint_{Q} \abs{Mu -(Mu)_{Q}} \dx \leq C \diam(Q)^{\alpha}  \inf _{z \in Q} M\mu (z)  \]
for all $Q$.
\end{theorem}

\begin{proof}
Take any cube $Q_0 = Q(x_0,r_0) \subset \mathbb{R}^{n}$ with side length $r_0$. Denote
\begin{align*}
U_1 (x) &= \sup \{u_Q :Q \in \mathcal{Q}(x),  \ell(Q) \leq r_0  \} \\
U_2 (x) &= \sup \{u_Q :Q \in \mathcal{Q}(x),  \ell(Q) > r_0 \} \\
U(x) 	&= \max (U_1(x),U_2(x)) .
\end{align*}
If $M$ is the centred maximal operator, then $\mathcal{Q}(x)$ is the collection of cubes centred at $x$. If $M$ is the non-centred maximal operator, then $Q \in \mathcal{Q}(x)$ whenever $x \in Q$. Our aim is to show that $U = Mu$ satisfies a Poincar\'e inequality. We reduce the task to controlling the quantity
\begin{align}
\label{eq:todistus1}
\dashint_{Q_0} \abs{U-U_{Q_0}} \dx &= 2 \dashint_{Q_0} (U-U_{Q_0})^{+} \dx \nonumber \\
 	&= \frac{2}{\abs{Q_0}}\int_{Q_0 \cap \{U_2 \leq U_1\}}(U_1 - U_{Q_0})^{+} \dx \nonumber \\
 	& \qquad + \frac{2}{\abs{Q_0}} \int_{Q_0 \cap \{U_2 > U_1\}}(U_2 - U_{Q_0})^{+} \dx   \nonumber \\
 	&= \frac{2}{\abs{Q_0}} (I + II) 
\end{align}
in two parts. The first term corresponds to local behaviour.

\subsection{Local part}
\label{sec:lopa}
Take the cube $3Q_0$ that is centred at $x_0$ and that has side length $3 r_0$. Every cube $Q$ with $Q \cap Q_0 \neq \varnothing$ and $\ell(Q) \leq r_0$ is contained in $3Q_0$. Consequently, we can regard $U_1$ as a maximal function localized in $3Q_0$, that is,
\begin{align*}
U_1(x) - u_{3Q_0} \leq M [1_{3Q_0}( u  - u_{3Q_0}  ) ](x)
\end{align*}
for every $x \in Q_0$. Remember that $u \geq 0$. Here $M$ is the non-centred Hardy-Littlewood maximal function. It is known to be bounded on $L^{q,\infty}$ (see for instance \cite{Grafakos2008}) so
\begin{align}
\norm{ 1_{Q_0} (U_1 - u_{3Q_0})}_{L^{q,\infty}(\mathbb{R}^{n})} 
& 	\leq \norm{ M( 1_{3Q_0}( u  - u_{3Q_0}  )) }_{L^{q,\infty}(\mathbb{R}^{n})}	\nonumber \\
& 	\lesssim \norm{ 1_{3Q_0}( u  - u_{3Q_0}  ) }_{L^{q,\infty}(\mathbb{R}^{n})} 	\nonumber \\
\label{eq:ref_fpw}
& 	\lesssim \diam(Q_0)^{\alpha} \frac{\mu(3 Q_0)}{|3 Q_0|^{1/q'}}.
\end{align}
where $q =  \frac{ n }{ n-\alpha }  > 1$ if $\alpha > 0$ . Inequality \eqref{eq:ref_fpw} comes from an application of Lemma \ref{lemma:fpw}. If $\alpha = 0$ and $\mu = 1$, then we have the same bound by the John-Nirenberg lemma. Hence
\begin{align}
\label{eq:todistus2}
\frac{I}{|Q_0|}	&	\leq \dashint_{Q_0} (U_1 - U_{Q_0})^{+} \dx \leq \dashint_{Q_0} (U_1 - u_{3Q_0})^{+} \dx +  (u_{3Q_0} - u_{Q_0})^{+}	\nonumber \\
& \lesssim	|Q_0|^{-1/q} \norm{ 1_{Q_0} (U_1 - u_{3Q_0})}_{L^{q,\infty}(\mathbb{R}^{n})}  + \dashint_{3Q_0} | u-   u_{3Q_0} | \dx \nonumber \\
& 	\lesssim \diam(Q_0)^{\alpha} \frac{\mu(3 Q_0)}{|3 Q_0|}
\end{align}
where we used the simple fact that
\[   \frac{1}{|E|^{1/q'}} \int_{E} |f| \dx \lesssim_q  \norm{f }_{L^{q, \infty}(\mathbb{R}^{n})}  \]
for all measurable sets $E$ that have finite measure.

Inequality \eqref{eq:todistus2} is the desired bound for the local part. Note that the proof is very abstract. It works for both centred and non-centred maximal functions as such. Transition to a maximal function using balls does not make any difference either.

\subsection{Preparation for the non-local part}
\label{sec:nolopa}
In the non-local part, we have to implement a construction that differs a bit for the centred and non-centred maximal functions. This is the only part of the proof where the two cases diverge, and it is packed into this subsection whose output is uniform for both cases. Our aim is to establish a pointwise bound for
\[(U_2(x) - U(y))^{+}, \quad x,y \in Q_0.\]
We proceed by proving a uniform bound in the case where $U_2(x)$ is replaced by an average over any admissible cube and $U(y)$ is estimated from below by a suitable cube. We start with the non-centred case. The construction is illustrated in Figure \ref{fig1:non-local}.

\textit{Non-centred case.}
Take any $Q_1$ such that $\ell (Q_1) \geq r_0$ and $Q_1 \cap Q_0 \neq \varnothing$. Let $Q_2 = Q_1 + h$ where $h \in \mathbb{R}^{n}$ is such that $Q_2 \supset Q_0$ and $|Q_1 \cap Q_2|$ is maximized. Now
\begin{align*}
(u_{Q_1} - u_{Q_2})^{+}  &= \frac{1}{|Q_1|}\left(\int_{Q_1 \setminus Q_2} u \dx  - \int_{Q_2 \setminus Q_1} u \dx \right)^{+} \\
& \leq  \frac{1}{|Q_1|}  \left \lvert \int_{A} \varphi u \dx \right \rvert
\end{align*} 
where $\varphi = 1_{Q_2 \setminus Q_1} - 1_{Q_1 \setminus Q_2}$ obviously has mean zero and $A$ is an annulus 
\[A = Q(x',\ell') \setminus Q(x'',\ell'') = Q' \setminus Q'' \supset (Q_2 \setminus Q_1) \cup (Q_1 \setminus Q_2).\] 
We can choose the annulus such that $\ell' - \ell'' = 2 r_0$ and $\ell' \leq 10 \ell (Q_1)$. Consequently,
\[\frac{\abs{A}}{\abs{Q_1}} \leq C \frac{r_0}{\ell (Q_1)}.\]
Note also that by the fact that $\varphi$ has mean zero in $A$, we have that 
\begin{align}
\label{eq:case1}
\left \lvert \int_{A} \varphi u \dx \right \rvert &\leq \int_{A} \abs{u-u_{Q_{*}} } \dx 
\end{align}
for any constant $u_{Q_*}$. We choose $Q_*$ to be the central cube obtained by dividing each side of $Q'$ into three equal parts. The data $A,Q_*,Q_1$ will be needed in completing the proof. 
Before proceeding to that, we construct these objects in the centred setting.

\begin{figure}
\includegraphics[scale=0.8]{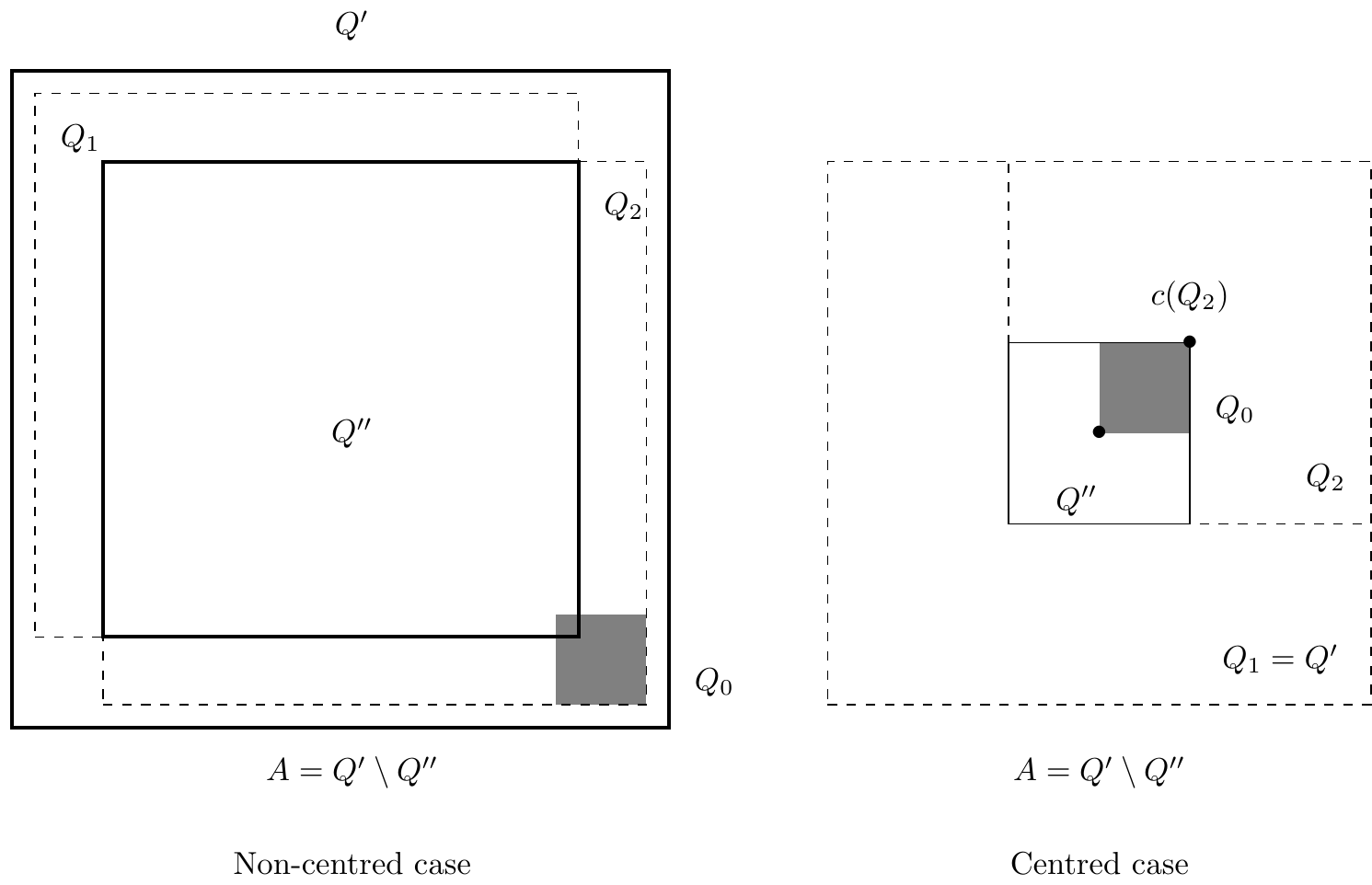}
\caption{\label{fig1:non-local} Construction of the annulus $A$ for both the centred and non-centred cases.} 
\end{figure}

\textit{Centred case.}
We modify the previous argument for the centred maximal operator. We also do it so that the argument carries over to the domain setting to be discussed later. We take again $Q_1$ that is centred at a point $y \in Q_0$. If $\ell (Q_1) < 10 r_0$, we can estimate
\[(u_{Q_1} - u_{Q_2})^{+} \lesssim r_0^{\alpha} \frac{\mu(Q_1)}{|Q_1|} \lesssim r_0^{\alpha} \inf_{z \in Q_0} M \mu (z)  \]
for any $Q_2$ centred at a point $x \in Q_0$ and $\ell(Q_2) = r_0 $. Taking the supremum over $Q_1$ we get the correct bound pointwise for
\[(U_2(y) - U_{Q_0} )^{+} \leq (U_2(y) - \inf_{z \in Q_0} U(z) )^{+}  ,\]
and we are done.

Assume then that $\ell(Q_1)  \geq 10 r_0$. Note that $Q_0 \subset Q_1$. Take $z \in Q_0$. We choose the maximal cube $Q_2$ centred at $z$ such that $Q_2 \subset Q_1$. Then 
\begin{align*} 
(u_{Q_1} - u_{Q_2})^{+} &= \frac{1}{|Q_1|} \left( \int_{Q_1} u \dx - \int_{Q_2} u \dx + (|Q_2|-|Q_1|)u_{Q_2} \right)^{+} \\
&= \frac{1}{|Q_1|} \left( \int_{Q_1 \setminus Q_2} u \dx - (|Q_1 \setminus Q_2|)u_{Q_2} \right)^{+} \\
&\leq \frac{1}{|Q_1|} \int_{Q_1 \setminus Q_2} (u-u_{Q_2})^{+} \dx.
\end{align*} 
We can compute further
\begin{align}
\label{eq:case2}
\frac{1}{|Q_1|} \int_{Q_1 \setminus Q_2} (u-u_{Q_2})^{+} \dx &\leq  \frac{1}{|Q_1|} \int_{Q_1 \setminus Q_2} (u-u_{Q_{*}})^{+} \dx    \\
& \qquad + \frac{|Q_1 \setminus Q_2|}{|Q_1|}  (u_{Q_{*}}-u_{Q_2})^{+} . \nonumber
\end{align}
The choice of $Q^{*}$ is done as follows. In the first term, note that $Q_1 \setminus Q_2$ is contained in an annulus $A = Q' \setminus Q''$ such that $\ell(Q') - \ell(Q'')  =  2 r_0$. We can choose $Q_1 = Q'$ in this case. Later in connection Corollary \ref{thm:domain}, it will be crucial that we can choose $A \subset Q_1$. We let $Q_{*}$ be the central triadic subcube of $Q'$. The second term is then clear by the fact that
\[\frac{|Q_1 \setminus Q_2|}{|Q_1|}  (u_{Q_{*}}-u_{Q_2})^{+} \lesssim \frac{r_0}{\ell(Q_1)} \diam(Q_2)^{\alpha} \frac{\mu(Q_2)}{|Q_2|} \lesssim r_0^{\alpha}  M \mu (x_0) \] 
where we used that $r_0 \lesssim \diam(Q_2) \eqsim l(Q_1)$ and $\alpha \in [0,1]$. Hence we have reduced the proof to estimating \eqref{eq:case1} or equivalently the first term in \eqref{eq:case2}:
\[\frac{1}{|Q_1|} \int_{A} (u - u_{Q_*})^{+}. \]

\subsection{Whitney decomposition}
We form a Whitney decomposition of $Q'$. Namely, we first divide each side by $3$. We choose the middle cube to $\mathcal{W}_1$. This was previously called $Q_*$. We subdivide the remaining cubes by halving each side length and the resulting cubes that touch the element of $\mathcal{W}_1$ form $\mathcal{W}_2$. Continuing inductively by dividing always the side length by $2$, we get a collection of cubes $\mathcal{W} = \cup_{i=1}^{\infty} \cup \mathcal{W}_i$ with the properties
\begin{align*}
Q \cap P &\in \{\varnothing, Q, P\}, \quad Q,P \in \mathcal{W} ,\\
\ell(Q) &=  d(Q , \partial Q') ,  \\
\cup_{Q \in \mathcal{W}} Q &= Q' \quad \textrm{up to a set of measure zero.}
\end{align*}
Let $\mathcal{A} = \{Q \in \mathcal{W} : Q \cap A \neq \varnothing \}$.


For each $Q \in \mathcal{A}$, we join the center $c(Q)$ to the center of $Q'$ by a straight line $\gamma$. In the center, there is the Whitney cube $Q_*$. Let 
\[\mathcal{C}(Q) = \{Q \in \mathcal{W}:   Q \cap \gamma \neq \varnothing \} .\]
For all $Q \in \mathcal{A}$, we write $\mathcal{C}(Q) = \{Q^i\}_{i=1}^{m_Q}$ where the cubes are ordered according to their distance from the point $c(Q)$. Note that 
\begin{align*}
| (2Q^i) \cap (2Q^{i+1}) | \eqsim_n |Q^i| ,\\
Q^1 = Q \quad \textrm{and} \quad Q^{m_Q} = Q_* .
\end{align*}

The previous construction at hand, we are ready to estimate
\begin{align*}
 \int_{A} \abs{u-u_{Q_{*}} } \dx \leq \sum_{Q \in \mathcal{A} } \int _{Q} \abs{u-u_{Q_*}} \dx .
\end{align*}
By the property $|2Q^i \cap 2Q^{i+1} | \eqsim_n |Q^i| $ there is $R^i \subset 2Q^i \cap 2Q^{i+1}$ with $|R^i| \eqsim_n Q^i$ so that
\begin{align*}
\int _{Q} |u-&u_{Q_*}| \dx \leq  \int _{Q} \abs{u-u_{Q}} \dx +  \abs{Q}  \sum_{j = 1}^{m_Q-1} |u_{Q^i} - u_{Q^{i+1}}| \\
 &\leq   \int _{Q} \abs{u-u_{Q}} \dx +  \abs{Q}  \sum_{j = 1}^{m_Q-1} \dashint_{R^i} (|u_{Q^i}-u  |+ |u - u_{Q^{i+1}}|) \dx \\
 &\lesssim \abs{Q}  \sum_{P \in \mathcal{C}(Q) } \diam (P)^{\alpha}  \frac{\mu(P)}{|P|} \\
 &= \abs{Q}  \sum_{P \in \mathcal{C}(Q) } |P|^{\alpha/n-1}   \mu(P)
\end{align*} 
Let $\mathcal{S(P)} = \{Q \in \mathcal{A}: P \in \mathcal{C}(Q) \}$. Notice the following:
\begin{enumerate}
\item[(i)] If $P \in \mathcal{A}$, then $Q \subset c_n P$ for a dimensional constant $c_n$ whenever $Q \in \mathcal{S}(P)$. This follows from the fact that $\mathcal{C}(Q)$ cannot contain cubes of equal size more than a uniformly bounded amount.
\item[(ii)] If $P \in \mathcal{W} \setminus \mathcal{A}$, then $Q \subset A_P = (c_n P) \cap A$ for a dimensional constant $c_n$ whenever $Q \in \mathcal{S}(P)$.
\end{enumerate}

We change the order of summation in the previous estimate to get
\begin{align*}
\sum_{Q \in \mathcal{A} } \int _{Q} \abs{u-u_{Q_*}} \dx & \lesssim  \sum_{Q \in \mathcal{A} } \abs{Q}  \sum_{P \in \mathcal{C}(Q) } |P|^{\alpha/n-1}   \mu(P)\\
& \leq  \sum_{P \in \mathcal{W} } \sum_{Q \in \mathcal{S}(P) } \abs{Q}   |P|^{\alpha/n-1}   \mu(P)\\
& = \sum_{P \in \mathcal{W} \cap \mathcal{A} }   \sum_{Q \in \mathcal{S}(P) } \abs{Q}   |P|^{\alpha/n-1}   \mu(P) \\
& \qquad + \sum_{P \in \mathcal{W} \setminus \mathcal{A} }   \sum_{Q \in \mathcal{S}(P) } \abs{Q}   |P|^{\alpha/n-1}   \mu(P)  \\
&= S_1 + S_2 .
\end{align*}
To estimate the the term $S_1$, we note that by (i)
\[\sum_{P \in \mathcal{W} \cap \mathcal{A} } \sum_{Q \in \mathcal{S}(P) } \abs{Q}   |P|^{\alpha/n-1} \mu(P) \lesssim \sum_{P \in \mathcal{W} \cap \mathcal{A} } |P|^{\alpha/n}  \mu(P). \]
In term $S_2$ we use the fact (ii) and an estimate for the volume of a cube intersected with an annulus. We obtain the bound
\begin{align*}
\sum_{P \in \mathcal{W} \setminus \mathcal{A} }   \sum_{Q \in \mathcal{S}(P) } \abs{Q}   |P|^{\alpha/n-1}  \mu(P)& \leq \sum_{P \in \mathcal{W} \setminus \mathcal{A} }  |A_P|    |P|^{\alpha/n-1}  \mu(P)  \\
&\lesssim \sum_{P \in \mathcal{W} \setminus \mathcal{A} }   \frac{r_0 |P|}{|P|^{1/n}} \cdot    |P|^{\alpha/n-1}  \mu(P) \\
& = r_0^{\alpha} \sum_{P \in \mathcal{W} \setminus \mathcal{A} } \left( \frac{|P|^{1/n}}{r_0}  \right)^{\alpha -1}   \mu(P) \\
& \leq r_0^{\alpha} \sum_{P \in \mathcal{W} \setminus \mathcal{A} }  \mu(P)
\end{align*}
since $\alpha \in [0,1]$ and $r_0 \leq |P|^{1/n}$ for $P \in \mathcal{W} \setminus \mathcal{A} $. Hence
\[S_1 + S_2 \lesssim r_0^{\alpha} \mu(Q') .\]
The constants are uniform in the cubes $Q_1$ and $Q_2$ so we can conclude that
\begin{equation}
\label{eq:todistus3}
\dashint_{Q_0} (U_2 - U_{Q_0})^{+} \dx \lesssim \diam (Q_0)^{\alpha} \inf _{z \in Q_0} M\mu(z) .
\end{equation}
This completes the proof of the bound for term $II$, both in the centred and the non-centred case.

\subsection{Final estimate}
Completing the estimate \eqref{eq:todistus1} by using \eqref{eq:todistus2} and \eqref{eq:todistus3} we reach the desired Poincar\'e inequality
\[ \dashint_{Q} \abs{Mu -(Mu)_{Q}} \dx \leq C \diam(Q)^{\alpha} \inf_{z \in Q} M\mu(z)  .\]
Note that the maximal function on the right hand side is whichever we prefer to use since the centred and the non-centred maximal functions are pointwise comparable. 
\end{proof}

\begin{remark}
It is also possible to run the previous proof for the maximal function defined using balls. Even if the proof seems to heavily rely on the properties of cubes, most cubes appearing in the proof are not related to the maximal function. In the case of the maximal function using balls, the local step would be almost identical. In the non-local step,  we have to use the Whitney decomposition of a ball instead of that of a cube, but are no essential differences. The argument for the centred maximal function also works in the setting where the centred maximal function is restricted to a proper subdomain of $\mathbb{R}^{n}$. Consequences of this extension are discussed in Corollary \ref{thm:domain}. The argument above does not work as such if we try to bound non-centred maximal function in a domain.
\end{remark}

\begin{remark}
We can also write the inequality resulting from Theorem \ref{theorem:preserved} as
\[ \dashint_{Q} \abs{Mu -(Mu)_{Q}} \dx \leq C \diam(Q)^{\alpha} \cdot \sup_{Q' \supset Q} \frac{\mu(Q')}{|Q'|}   , \]
and this can be verified using the same proof with a more careful bookkeeping. It is also possible to interpret the quantities appearing in the proof in the spirit of the representation formulas of \cite{Luiro2007}.
\end{remark}

Next we point out how several classical results can be deduced from Theorem \ref{theorem:preserved}. Consider functions $u \in W^{1,p}(\mathbb{R}^{n})$ where $p \in (0,1)$. Since
\[\inf _{z \in Q} M|\nabla u|(z) \leq \left( \dashint_{Q} M(|\nabla u|)^{p} \dx \right) ^{1/p} \]
for all $p \in (1, \infty)$, our Theorem \ref{theorem:preserved} and Lemma 6 in Haj\l{}asz \cite{Hajlasz2003} give a new proof for the boundedness result $M:W^{1,p}(\mathbb{R}^{n})  \to W^{1,p}(\mathbb{R}^{n})$. This was originally proved by Kinnunen in \cite{Kinnunen1997}. See also \cite{HO2004}.

\begin{corollary}[Kinnunen \cite{Kinnunen1997}]
\label{cor:kinnunen1997}
Let $p > 1$. Then the Hardy-Littlewood maximal function is bounded $M:W^{1,p}(\mathbb{R}^{n}) \to W^{1,p}(\mathbb{R}^{n})$. In addition,
\[|\nabla Mu| \lesssim M |\nabla u|.\] 
\end{corollary} 
\begin{proof}
Note that $|\nabla u| = |\nabla |u||$. Then $|u|$ and $|\nabla u|$ satisfy a Poincar\'e inequality with $\alpha = 1$. Since $u \in L^{p}(\mathbb{R}^{n})$, $Mu$ is locally integrable. By Theorem \ref{theorem:preserved}, $Mu$ and $M|\nabla u|$ satisfy a Poincar\'e inequality with $\alpha = 1$. By Lemma 6 in Haj\l{}asz \cite{Hajlasz2003} the claim follows.
\end{proof}

Specializing to $\mu \in L^{\infty}(\mathbb{R}^{n})$ and lowering the value of $\alpha$ in the Poincar\'e inequality, we end up studying spaces of H\"older continuous functions. We get the following corollary. See also \cite{CF1987} for a related result.
\begin{corollary}[Buckley \cite{Buckley1999}]
Let $u \in \Lambda(\alpha)$ with $\alpha \in (0,1]$ be such that $Mu \in L^{1}_{loc}(\mathbb{R}^{n})$. Then $Mu \in \Lambda(\alpha)$ with the same $\alpha$.  
\end{corollary}
\begin{proof}
The space $\Lambda(\alpha)$ consist of the measurable functions that have an $\alpha$-H\"older continuous representative. It is characterized by a Poincar\'e inequality of $u$ and $\mu = 1$ with exponent $\alpha$ for the diameter (see \cite{Meyers1964}). By Theorem \ref{theorem:preserved}, the maximal function preserves this Poincar\'e inequality provided that $Mu$ is locally integrable. Hence the claim follows. 
\end{proof}

Finally, looking at $\alpha = 0$ and $\mu \in L^{\infty}(\mathbb{R}^{n})$, we can include $\BMO$. It is interesting that we can cover both $\BMO$ and Sobolev spaces with a single proof. Finding such a proof was our original motivation for writing this note about Sobolev spaces, Poincar\'e inequalities, and the maximal function.

\begin{corollary}[Bennett, DeVore, and Sharpley \cite{BdVS1981}]
Let $u \in \BMO$ with $Mu \in L^{1}_{loc}(\mathbb{R}^{n})$. Then 
\[\norm{Mu}_{\BMO} \lesssim \norm{u}_{\BMO} .\]
\end{corollary}

The claim follows directly from Theorem \ref{theorem:preserved}. On the other hand, the proof of Theorem \ref{theorem:preserved} also contains a simplification of the original argument of \cite{BdVS1981}. Since this might be of independent interest, we write down another simple proof of this corollary. We need not use the $A_p$ weights like for instance the proof in \cite{CF1987}.

\begin{proof}[Another simple proof]
Let $M$ be the non-centred maximal function. Since $|u| \in \BMO$, we can assume that $u \geq 0$. Take a cube $Q$. Let $E = Q \cap \{Mu = M(1_{3Q} u)\}$. Then
\begin{align*}
&\int_{Q} | Mu - (Mu)_{Q} | \dx  =  2\int_{Q} ( Mu - (Mu)_{Q} ) ^{+} \dx \\
	& = 2  \int_{E} ( M(1_{3Q}u) - (Mu)_{Q} ) ^{+}\dx 
	 + 2  \int_{Q \setminus E} ( Mu - (Mu)_{Q} ) ^{+} \dx \\
	& \leq 2  \int_{Q} |M(1_{3Q}(u - u_{3Q} ))|\dx 
	 + 2  \int_{Q\setminus E} \sup_{\substack{ x \in P \\ |P| > |Q| }} | u_P - u_{2P} |\dx \\ 
	& = I + II.
\end{align*}
$P$ denotes a cube in the supremum. Clearly $II \lesssim  |Q| \norm{u}_{\BMO} $. By H\"older's inequality, $L^{2}$ boundedness of $M$ and the John-Nirenberg inequality
\begin{align*}
I  \lesssim |Q|^{1/2} \norm{M(1_{3Q}(u - u_{3Q} ))}_{L^{2}} 
	&\lesssim |Q|^{1/2} \norm{1_{3Q}(u - u_{3Q} )}_{L^{2}} \\
	&\lesssim |Q| \norm{u}_{\BMO}
\end{align*}
so the claim follows.
\end{proof}

\section{The endpoint Sobolev space}
Next we use Theorem \ref{theorem:preserved} in the endpoint space $W^{1,1}(\mathbb{R}^{n})$ to prove that the distributional partial derivatives of $Mu$ acting on test functions vanishing at a singularity set are actually functions in $L^{1,\infty}(\mathbb{R}^{n})$. Moreover, if we assume that $|\nabla u| \in L \log L (\mathbb{R}^{n})$, we recover local integrability of $\nabla Mu$. This was first proved by Haj\l{}asz and Onninen in \cite{HO2004}. Of course, we do not know if any of the results is optimal. Recall that no counterexample for $u \in W^{1,1}(\mathbb{R}^{n})$ implying $|\nabla Mu| \in L^{1}_{loc}$ is known at the moment. The validity of the boundedness of $u \mapsto  |\nabla Mu|$ from $W^{1,1}(\mathbb{R}^{n})$ to $L^{1}(\mathbb{R}^{n})$ is an interesting open question.

There are also other approaches to the maximal function of $u \in W^{1,1}(\mathbb{R}^{n})$. Haj\l{}asz and Mal\'y \cite{HM2010} show that $Mu$ is approximately differentiable almost everywhere, which implies that it coincides with a continuously differentiable function outside an open set of arbitrarily small measure. This property is weaker than weak differentiability. On the other hand, they prove the result in a more general setting under less stringent assumptions. See also the related paper by Luiro \cite{Luiro2013}.


\begin{corollary}
\label{thm:sobolev}
Let $u \in W^{1,1}(\mathbb{R}^{n})$. Then the distributional partial derivatives $\partial_i Mu$ can be represented as functions $h_i \in L^{1,\infty}(\mathbb{R}^{n})$ when they act on smooth functions with compact support not meeting
\[\{x \in \mathbb{R}^{n}: \liminf_{\delta \to 0} \int_{B(x,\delta)} M |\nabla u| \dy = \infty \}.\]
Moreover,
\[ |h_i| \lesssim M |\nabla u| , \quad i \in \{1,\ldots, n\} . \]
\end{corollary}
\begin{proof}
First of all, note that $|u|$ satisfies the assumptions of Theorem \ref{theorem:preserved}. The property $Mu \in L^{1}_{loc}(\mathbb{R}^{n})$ follows from the Sobolev embedding theorem. The function $Mu$ gives rise to a bounded linear functional acting on smooth functions with compact support. We study the distributional derivative defined by
\[ \partial_i Mu (\varphi) = -\int_{\mathbb{R}^{n}} Mu \cdot (\partial_i \varphi) \dx   \]
where $\varphi$ belongs to the space of smooth and compactly supported functions $C_0^{\infty}(\mathbb{R}^{n})$. The distributional derivative of $Mu$ need not arise from a locally integrable function since $L^{1,\infty}(\mathbb{R}^{n})$ is the best integrability that we can expect at the moment.

Let 
\[A = \{x \in \mathbb{R}^{n}: \liminf_{\delta \to 0} \int_{B(x,\delta)} M |\nabla u| \dy = \infty \}. \]
Note that $A$ is closed. Indeed, take any $x \in \overline{A}$. Take a sequence $A \ni x_i \to x$. Take $\epsilon > 0$. There is $i$ such that $x_i \in B(x,\epsilon)$. Since $x_i \in A$, there is $\delta > 0$ such that $B(x_i,\delta) \subset B(x,\epsilon)$ and 
\[ \int_{B(x_i,\delta)} M |\nabla u| \dy = \infty.  \]
Consequently 
\[\int_{B(x,\epsilon) }  M |\nabla u| \dy \geq \int_{B(x_i,\delta)} M |\nabla u| \dy = \infty \]
and $x \in A$. So $A$ is closed. We define $\mathcal{D}(A^{c})$ to consist of all test functions $\varphi \in C_0^{\infty}(\mathbb{R}^{n})$ that are supported outside $A$. When we talk about distributions, we refer to the dual of this test function class supported outside $A$. We prove that $\partial_i Mu$ on $\mathcal{D}(A^{c})$ is given by integration against $h \in L^{1,\infty}(\mathbb{R}^{n})$. 

To prove that $\partial_i Mu$ is a measurable function, we use the argument of Haj\l{}asz \cite{Hajlasz2003}. Let $\psi \in C_0^{\infty}(B(0,1))$ be a positive function with $\int_{\mathbb{R}^{n}} \psi \dx = 1$. For $\epsilon \in (0,1)$ we define the dilations $\psi_\epsilon (x) = \frac{1}{\epsilon^{n}} \psi (\frac{x}{\epsilon})$. We have that
\[\partial_i Mu = \lim_{\epsilon \to 0} (Mu * \partial_i \psi_{\epsilon}) \]
in the sense of distributions. Since $\int \partial_i \psi_{\epsilon}  = 0$, also
\[(Mu * \partial_i \psi_{\epsilon}) (y) = [(Mu - (Mu)_{B(y,\epsilon)}) * \partial_i \psi_{\epsilon}](y) . \] 
By Theorem \ref{theorem:preserved}
\[| (Mu * \partial_i \psi_{\epsilon}) (y) | \leq C \epsilon^{-n-1} \int_{B(y,\epsilon)} |Mu - (Mu)_{B(y,\epsilon)}| \dx \leq C M| \nabla u |(y). \]
Let
\[g_\epsilon = (Mu * \partial_i \psi_{\epsilon}) \quad \textrm{and} \quad g = \lim_{\epsilon \to 0}  g_\epsilon = \lim_{\epsilon \to 0} (Mu * \partial_i \psi_{\epsilon}),  \]
where the limit is again in the sense of distributions.

Since 
\[ |g (\varphi)| = \lim_{\epsilon \to 0} | g_\epsilon (\varphi)| \leq C \norm{\varphi }_{L^{\infty}} \int_{\supp \varphi} M| \nabla u | \dx ,  \]
$g$ is a bounded linear functional on $C_0(\mathbb{R}^{n} \setminus A)$. By the Riesz representation theorem it is a signed Radon measure. It is also absolutely continuous with respect to the Lebesgue measure. By the Radon-Nikodym theorem it is represented by integration against a measurable function $h$. This proves the existence of the distributional derivative as a measurable function.

Since 
\[\int \varphi \partial_i Mu \dx \lesssim \norm{\varphi}_{L^{\infty}} \int_{\supp \varphi} M | \nabla u |  \dx \]
holds for all test functions, we can conclude that $\partial_i Mu \lesssim M | \nabla u |$ almost everywhere. 
\end{proof}

The proof was based on Theorem \ref{theorem:preserved} which in turn holds whenever the modulus of the generalized gradient is a general Radon measure. It need not be locally integrable function. Hence we can extend the result to the class of $\BV$ functions. See \cite{EG1991} for the precise definitions. 

\begin{corollary}
Let $u \in \BV$. Then the distributional partial derivatives $\partial_i Mu$ can be represented as functions $h_i \in L^{1,\infty}(\mathbb{R}^{n})$ when they act on smooth functions with compact support not meeting
\[\{x \in \mathbb{R}^{n}: \liminf_{\delta \to 0} \int_{B(x,\delta)} M |\nabla u| \dy = \infty \}.\]
Moreover,
\[ |h_i| \lesssim M |\nabla u|, \quad i \in \{1,2,\ldots, n\}. \]
\end{corollary}

The result concerning the endpoint Sobolev space can be localized. By minor modifications in the arguments, we can prove that Corollary \ref{thm:sobolev} holds true for centred maximal operators in subdomains $\Omega \subset \mathbb{R}^{n}$. By a domain we mean an open and connected set. We define
\[M_\Omega u(x) = \sup_{Q_x \subset \Omega} \dashint_{Q_x} |u| \dx \]
where the supremum is over all cubes centred at $x$ and with sides parallel to coordinate axes. 

\begin{corollary}
\label{thm:domain}
Let $\Omega \subset \mathbb{R}^{n}$ be a domain, and let $u \in W^{1,1}(\Omega)$. Then the distributional partial derivatives $\partial_i M_\Omega u$ can be represented as functions in $h_i \in L^{1,\infty}(\Omega)$ when they act on smooth functions with compact support not meeting
\[\{x \in \Omega :\liminf_{\delta \to 0} \int_{B(x,\delta)} M_\Omega |\nabla u| \dy = \infty \}.\]
Moreover,
\[| h_i | \lesssim M_\Omega |\nabla u| ,\quad i \in \{1,2,\ldots, n\}. \]
\end{corollary}
\begin{proof}
We first note that Theorem \ref{theorem:preserved} carries over to this setting provided that we weaken the claimed inequality to only hold for $Q$ such that $20 Q \subset \Omega$. However, this is enough for the proof of Corollary \ref{thm:sobolev}. The claim follows from a verbatim repetition of its proof.
\end{proof}

The same method also applies to the case $p > 1$. Namely, we get a new proof for the theorem of Kinnunen and Lindqvist \cite{KL1998}. Even if the approaches of the original proofs were different, the proof via Poincar\'e inequalities does not see any difference between the full space and a proper domain as far as centred maximal functions are concerned.

\begin{corollary}[Kinnunen and Lindqvist \cite{KL1998}]
Let $\Omega \subset \mathbb{R}^{n}$ be a domain, and let $u \in W^{1,p}(\Omega)$. Then $M_\Omega u \in W^{1,p}(\Omega)$ and
\[| \nabla M_\Omega u | \lesssim M_\Omega |\nabla u| . \]
\end{corollary}
\begin{proof}
This is essentially same as the proof given for Corollary \ref{cor:kinnunen1997} as the proof of Corollary \ref{thm:domain} reveals.
\end{proof}

Note that it is necessary to use the centred maximal function here if we talk about cubes. Consider a domain that is a wide square at the end of a narrow corridor. A smooth bump in the middle of the square is a Sobolev function. However, its non-centred maximal function necessarily has a jump when we move from the corridor to the square. It is zero in the corridor, but when it enters the the square, it suddenly exceeds some fixed positive threshold value.

\section{The fractional maximal function}
In this concluding section, we make some remarks on how the method applies to the fractional maximal function. We define the non-centred fractional maximal function to be
\[M_\beta f = \sup_{Q \ni x} \ell(Q)^{\beta} \dashint_{Q} |f| \dx , \quad \beta \in [0,n]. \]
We will next inspect how the local part of the proof (Section \ref{sec:lopa}) behaves with the fractional maximal function. 

Take a cube $Q_0$ and a triple $(u,\mu,\alpha)$ as in the premises of Theorem \ref{theorem:preserved} except that now we assume that $M_{\beta}u \in L^{1}_{loc}(\mathbb{R}^{n})$ in addition to $M u \in L^{1}_{loc}(\mathbb{R}^{n})$. Let $E = Q_0 \cap \{M_{\beta}u = M_{\beta}(1_{3Q_0}u) \}$. We see that for all $x \in E$ it is clear that
\[ M_\beta u - (M_\beta u)_{Q_0} \leq M_\beta (1_{3Q_0}u) - \ell(3Q_0)^{\beta} u_{3Q_0} \lesssim \ell(Q_0)^{\beta}( M  (1_{3Q_0}u) - u_{3Q_0}  ) \]
where the rightmost quantity can be estimated as in the proof of Theorem \ref{theorem:preserved}. The outcome is that
\[\frac{1}{|Q_0|} \int_{E} (M_\beta u - (M_\beta u)_{Q_0} )^{+} \dx \lesssim \diam(Q_0)^{\alpha} \inf_{z \in Q_0} M_{\beta}\mu(z) .\]

In the non-local part (Section \ref{sec:nolopa}) the changes that occur are even more marginal. Namely, the relevant quantity 
\[(\ell(Q_1)^{\beta} u_{Q_1} - \ell(Q_2)^{\beta} u_{Q_2} )^{+} = \ell(Q_1)^{\beta}  (u_{Q_1} - u_{Q_2} )^{+} \]
reduces to the case already known modulo the factor $\ell(Q_1)^{\beta}$. This will be absorbed in the fractional maximal function appearing in the right side of the final estimate. 

Hence the proof of Theorem \ref{theorem:preserved} with the initial changes mentioned above gives the estimate 
\[ \dashint_{Q} \abs{M_\beta u -(M_\beta u)_{Q}} \dx \leq C \diam(Q)^{\alpha}  \inf _{z \in Q} M_\beta \mu (z)  \]
for the fractional maximal function provided that the minimal assumptions to make the quantities above well defined are fulfilled. One can use this inequality to derive a bound for the gradient of the fractional maximal function of a Sobolev function. This will be a repetition of the argument in \cite{Hajlasz2003} so we omit it.


\begin{thebibliography}{99}
\bibitem{BdVS1981}
C.~Bennett, R.A.~DeVore and R.~Sharpley, \emph{Weak $L^{\infty}$ and BMO}, Ann. of Math. \textbf{113} (1981), 601--611.

\bibitem{BKM2016}
L. Berkovits, J. Kinnunen and J.M. Martell, \emph{Oscillation estimates, self-improving results and good-lambda inequalities}, J. Funct. Anal. \textbf{270} (2016), 3559--3590.

\bibitem{BM2015}
F. Bernicot and J.M. Martell, \emph{Self-improving properties for abstract Poincaré type inequalities}, Trans. Amer. Math. Soc. \textbf{367} (2015), no. 7, 4793--4835. 

\bibitem{BCHP2012}
J. Bober, E. Carneiro, K. Hughes and L.B. Pierce, \emph{On a discrete version of Tanaka's theorem for maximal functions}, Proc. Amer. Math. Soc. \textbf{140} (2012), no. 5, 1669--1680. 

\bibitem{Buckley1999}
S.M.~Buckley, \emph{Is the maximal function of a Lipschitz function continuous?}, Ann. Acad. Sci. Fenn. Math. \textbf{24} (1999), no. 2, 519--528.


\bibitem{Campanato1963}
S. Campanato, \emph{Propriet\`a di h\"olderianit\`a di alcune classi di funzioni},
Ann. Scuola Norm. Sup. Pisa (3) \textbf{17} (1963), 175--188. 

\bibitem{CH2012}
E. Carneiro and K. Hughes, \emph{On the endpoint regularity of discrete maximal operators},
Math. Res. Lett. \textbf{19} (2012), no. 6, 1245--1262.

\bibitem{CFS2015}
E. Carneiro, R. Finder and M. Sousa, \emph{On the variation of maximal operators of convolution type II}, available at arXiv:1512.02715 (2015).

\bibitem{CM2015}
E. Carneiro and J. Madrid, \emph{Derivative bounds for fractional maximal functions}, available at arXiv:1510.02965 (2015), to appear in Trans. Amer. Math. Soc.

\bibitem{CS2013}
E. Carneiro and B.F.~Svaiter, \emph{On the variation of maximal operators of convolution type}, J. Funct. Anal. \textbf{265} (2013), no. 5, 837--865.

\bibitem{CF1987}
F. Chiarenza and M. Frasca, \emph{Morrey spaces and Hardy-Littlewood maximal function}, Rend. Mat. Appl. (7) \textbf{7} (1987), no. 3-4, 273--279 (1988). 

\bibitem{EG1991} 
L.C. Evans and R.F. Gariepy, \emph{Measure Theory and Fine Properties of Functions}, CRC Press, 1991. 

\bibitem{FPW1998}
B. Franchi, C. P\'erez, R.L. Wheeden, \emph{Self-improving properties of John-Nirenberg and Poincar\'e inequalities on spaces of homogeneous type}, J. Funct. Anal. \textbf{153} (1998), no. 1, 108--146.

\bibitem{Grafakos2008}
L. Grafakos, \emph{Classical Fourier Analysis}, Springer, 2008.

\bibitem{Hajlasz2003}
P. Haj\l{}asz, \emph{A new characterization of the Sobolev space}, Studia Math. \textbf{159} (2003), no.~2, 263--275.

\bibitem{HK2000}
P. Haj\l{}asz and P. Koskela, \emph{Sobolev met Poincar\'e}, Mem. Amer. Math. Soc. \textbf{145} (2000), no. 688.

\bibitem{HM2010}
P. Haj\l{}asz and J. Mal\'y, \emph{On approximate differentiability of the maximal function},
Proc. Amer. Math. Soc. 138 (2010), no. 1, 165--174. 

\bibitem{HO2004}
P. Haj\l{}asz and J. Onninen, \emph{On boundedness of maximal functions in Sobolev spaces}, Ann. Acad. Sci. Fenn. Math. \textbf{29} (2004), no. 1, 167--176. 

\bibitem{HKKT2015}
T. Heikkinen, J. Kinnunen, J. Korvenp\"a\"a and H. Tuominen, \emph{Regularity of the local fractional maximal function}, Ark. Mat. \textbf{53} (2015), 127--154.

\bibitem{HLT2013}
T. Heikkinen, J. Lehrb\"ack, J. Nuutinen and H. Tuominen, \emph{Fractional maximal functions in metric measure spaces}, Anal. Geom. Metr. Spaces \textbf{1} (2013), 147--162. 

\bibitem{HMV2014}
R. Hurri-Syrj\"anen, N. Marola and A.V. V\"ah\"akangas, \emph{Aspects of local-to-global results}, Bull. Lond. Math. Soc. \textbf{46} (2014), no. 5, 1032--1042.

\bibitem{JM2013}
A. Jim\'enez-del-Toro and J.M. Martell, \emph{Self-improvement of Poincar\'e type inequalities associated with approximations of the identity and semigroups}, Potential Anal. \textbf{38} (2013), no. 3, 805--841.

\bibitem{Kinnunen1997}
J. Kinnunen, \emph{The Hardy-Littlewood maximal function of a Sobolev function},
Israel J. Math. \textbf{100} (1997), 117--124.

\bibitem{KL1998}
J. Kinnunen and P. Lindqvist, \emph{The derivative of the maximal function}, J. Reine Angew. Math. \textbf{503} (1998), 161--167. 

\bibitem{KS2003}
J. Kinnunen and E. Saksman, \emph{Regularity of the fractional maximal function}, Bull. London Math. Soc. \textbf{35} (2003), no. 4, 529--535.

\bibitem{Korry2002}
S. Korry, \emph{Boundedness of Hardy-Littlewood maximal operator in the framework of Lizorkin-Triebel spaces}, Rev. Mat. Complut. \textbf{15} (2002), 401--416. 

\bibitem{Kurka2010}
O. Kurka, \emph{On the variation of the Hardy-Littlewood maximal function},
Ann. Acad. Sci. Fenn. Math. \textbf{40} (2015), no. 1, 109--133. 

\bibitem{Luiro2007}
H. Luiro, \emph{Continuity of the maximal operator in Sobolev spaces}, Proc. Amer. Math. Soc. \textbf{135} (2007), no. 1, 243--251.

\bibitem{Luiro2010}
H. Luiro, \emph{On the regularity of the Hardy-Littlewood maximal operator on subdomains of $\mathbb{R}^{n}$}, Proc. Edinb. Math. Soc. (2) \textbf{53} (2010), no. 1, 211--237. 

\bibitem{Luiro2013}
H. Luiro, \emph{On the differentiability of directionally differentiable functions and applications}, available at arXiv:1208.3971 (2012).

\bibitem{macmanus2002}
P. MacManus, \emph{Poincar\'e inequalities and Sobolev spaces}, in Proceedings of the 6th International Conference on Harmonic Analysis and Partial Differential Equations (El Escorial, 2000), Publ. Mat. 2002, Vol. Extra, 181--197.

\bibitem{MP1998}
P. MacManus and C. P\'erez, \emph{Generalized Poincar\'e inequalities: sharp self-improving properties}, Internat. Math. Res. Notices 1998, no. 2, 101--116.

\bibitem{MS2016}
N. Marola and O. Saari, \emph{Local to global results for spaces of BMO type}, Math. Z. \textbf{282} (2016), no. 1-2, 473--484.


\bibitem{Meyers1964}
N.G. Meyers, \emph{Mean oscillation over cubes and H\"older continuity}, 
Proc. Amer. Math. Soc. \textbf{15} (1964), 717--721. 

\bibitem{Miranda2003}
M. Miranda Jr, \emph{Functions of bounded variation on ``good'' metric spaces}, J. Math. Pures Appl. (9) \textbf{82} (2003), no. 8, 975--1004. 


\bibitem{Saari2016}
O. Saari, \emph{Parabolic BMO and the forward-in-time maximal operator}, available at arXiv:1603.04452 (2016).

\bibitem{Tanaka2002}
H. Tanaka, \emph{A remark on the derivative of the one-dimensional Hardy-Littlewood maximal function}, Bull. Austral. Math. Soc. \textbf{65} (2002), no. 2, 253--258.

 
\end{thebibliography}
\end{document}